\documentclass[12pt]{amsart}

\usepackage[margin=2.3cm]{geometry} 
\usepackage[T1]{fontenc}
\usepackage[utf8]{inputenc}
\usepackage{lmodern}
\usepackage{amsmath,amssymb,amsthm,mathtools}
\usepackage{enumitem}
\usepackage[colorlinks=true,linkcolor=blue,citecolor=blue,urlcolor=blue]{hyperref}
\usepackage[nameinlink,capitalise]{cleveref}

\newtheorem{theorem}{Theorem}

\newtheorem{proposition}{Proposition}
\newtheorem{definition}{Definition}

\DeclareMathOperator{\Coin}{Coin}

\DeclareMathOperator{\N}{\mathbb{N}}
\DeclareMathOperator{\Z}{\mathbb{Z}}
\DeclareMathOperator{\R}{\mathbb{R}}

\title[A Counterexample to Horn's Conjecture]{A Counterexample to Horn's Coincidence Conjecture}

\author{Faruk Temur}
\date{}

\subjclass[2020]{Primary 47H10; Secondary 54H25}
\keywords{Horn's conjecture, coincidence point, commuting maps, compact convex set, finite tree, absolute retract, Hilbert space}

\begin{document}

\begin{abstract} Horn conjectured that whenever $K$ is a nonempty compact convex subset of a Banach space and $F,G:K\to K$ are continuous commuting maps, there exists $x\in K$ such that $F(x)=G(x)$. We give a counterexample in a separable Hilbert space. The construction uses an exactly commuting ladder of coincidence-free maps between finite trees, due to Oversteegen and Rogers. 
\end{abstract}

\maketitle

\section{Introduction}
In 1970, Horn \cite{Horn1970} posed the following conjecture in connection with asymptotic fixed point theory.

\vspace{2mm}

{\bf Horn's Conjecture} \emph{
Let $K$ be a nonempty compact convex subset of a Banach space, and let $F,G:K\to K$ be continuous maps satisfying $F\circ G=G\circ F$. Then there exists $x\in K$ such that $F(x)=G(x).$ }

\vspace{2mm}

The statement continued to be recorded as Horn's conjecture in later  literature; see, for example, \cite{ Rus2018, Sine}. Horn \cite{Horn1970} proves this for subsets of $\R$.  The obvious next target is to prove it in the plane, and attempts here concentrated   on the simplest non-trivial cases such as the disk and the simple triod, see \cite{Hagopian,McDowell2009}. However even these  turned out to be  rather difficult. It has been known, see  \cite{Sine}, that if true, Horn's conjecture implies the following result of R. Cauty \cite{Cauty} in  asymptotic fixed point theory. Let $E$ be a Banach space, and $A \subset E$ a closed, bounded, and convex subset. If
 $f : A \rightarrow A$ is  continuous such that  $f^m$ is compact for some $m\in \N$,  then $f$
has at least one fixed point.
 For a long time this was a conjecture attributed to Schauder, until it was proved in greater generality by R. Cauty, see \cite{Cauty,Nussbaum,Sine,Steinlein}.

In this note is we   give a negative answer to Horn's conjecture in infinite dimensions, specifically in a separable Hilbert space. The starting  point of the counterexample is from the work of Oversteegen and Rogers \cite{OversteegenRogers1980,OversteegenRogers1982}.  They construct,  for each nonnegative integer $n$,  finite trees $T_n$ and maps
$
 f_n,g_n:T_{n+1}\rightarrow T_n,
$
satisfying the exact ladder identities
$
 f_n\circ g_{n+1}=g_n\circ f_{n+1}, 
$
with the coincidence set of 
$
 f_0,g_0
$
being empty.
Their original use of this system is to induce a fixed-point-free map on a tree-like inverse limit.
Our  observation made here is that the same exact ladder can be converted into two commuting self-maps of a compact convex set that do not coincide. See also \cite{HernandezHoehn2018} for a clear exposition and improvements upon \cite{OversteegenRogers1980,OversteegenRogers1982}.
Our main result is the following.  The notation  $\Z_+$ stands for nonnegative integers.

\begin{theorem}\label{thm:main}
There exist a nonempty compact convex subset
$
 K\subset \ell^2(\Z_+;\R^2)
$
and continuous maps $F,G:K\to K$ such that
$
 F\circ G=G\circ F,
$
but
$
 F(x)\neq G(x)
 $  for every   $x\in K.$
Consequently, Horn's conjecture is false.
\end{theorem}

As  Horn's conjecture    is stated only with most basic  concepts of functional analysis, it may attract attention of mathematicians of diverse interests. Therefore, to make the  counterexample intelligible to  as wide an audience as possible, we will introduce and explain the concepts we use from continuum theory and retract theory, upon which the counterexample is built. The next section is devoted to this task.  Then in the third section  
we prove Theorem 1.

\section{Preliminaries and terminology}\label{sec:preliminaries}

This section records the terminology used in the proof. We assume the reader to be familiar with basic concepts of functional analysis, such as metric spaces, the Banach spaces etc.   For
continuum theory  we refer to the
standard graduate text \cite{Nadler1992}; and for retract theory we refer to the
foundational monograph \cite{Borsuk1967}.   Unless stated
otherwise, all vector spaces are real, all topological spaces considered below
are metrizable, and every map between topological spaces is assumed to be
continuous.  The notation $\Z_+$ stands for  nonnegative integers.

\subsection{Metric and linear terminology}

\begin{definition}[Connectedness and local connectedness]
A topological space $X$ is {connected} if it cannot be written as
$X=U\cup V$, where $U$ and $V$ are disjoint, nonempty, open subsets of $X$.
It is {locally connected} if for every point $x\in X$, and every open set $U$ containing $x$, there is an open connected set $V$ satisfying $x\in V\subseteq U.$
\end{definition}

\begin{definition}[Convexity]
Let $E$ be a vector space.  A subset $K\subset E$ is {convex} if
\begin{equation*}
 (1-t)x+ty\in K \qquad  \text{whenever}
 \qquad   x,y\in K,\ 0\leq t\leq1.
\end{equation*}
\end{definition}

\begin{definition}[The Hilbert direct sum]
 Let $E_n, \ n \in \Z_+$ be  a family of Hilbert spaces, with inner products $\langle \cdot,\cdot \rangle_n$. 
Then 
\begin{equation*}
	\ell^2(\Z_+;E_n)=\left(\bigoplus_{n=0}^{\infty}E_n\right)_{\ell^2}
	:=
	\left\{x=(x_n)_{n\in \Z_+}: \  x_n\in E_n,\ 
	\sum_{n=0}^{\infty}\lVert x_n\rVert_n^2<\infty\right\}
\end{equation*}
is their Hilbert direct sum, which is  itself a Hilbert space with coordinatewise addition and scalar multiplication, and the inner product
\begin{equation*}
\langle x,y \rangle:=\sum_{n\in \Z_+} \langle x_n,y_n \rangle_n.
\end{equation*}
\end{definition}

\subsection{Maps, coincidence points, exactly commuting ladders}

\begin{definition}[Maps and composition]
A {self-map} of a space $X$ is a map $F:X\to X$.  The identity self-map
is denoted by $\operatorname{id}_X$.  
Two self-maps $F,G:X\to X$ commute if
$
 F\circ G=G\circ F.
$
\end{definition}

\begin{definition}[Coincidence points]
Let $F,G:X\to X$. A point $x$ is a coincidence point of $F$ and $G$ if $F(x)=G(x)$.
 The coincidence set is
\begin{equation*}
	\Coin(F,G):=\{x\in X:F(x)=G(x)\}.
\end{equation*}
\end{definition}

\begin{definition}[Homeomorphisms and embeddings]
A map $h:X\to Y$ is a {homeomorphism} if it is bijective and both $h$
and $h^{-1}$ are continuous.  A map $e:X\to Y$ is a {topological
embedding} if it is a homeomorphism from $X$ onto the subspace $e(X)\subset
Y$.  A space is {planar} if it admits a topological embedding into
$\R^2$.  A continuous injection from a compact space into a Hausdorff space is
a topological embedding, and its image is closed.
\end{definition}

\begin{definition}[Exact commuting ladder]
	An {exact commuting ladder} is a sequence of spaces $T_n, n\in\Z_+$ together with
	two maps at each level,
	$
	f_n,g_n:T_{n+1}\rightarrow T_n,
	$
	such that for each $n\in \Z_+$ we have the commutation relation
	\begin{equation}\label{eq:definition-ladder}
		f_n\circ g_{n+1}=g_n\circ f_{n+1}.
	\end{equation}
	The
	ladder is {coincidence-free at level $n$} if
	$\Coin(f_n,g_n)=\varnothing$.
\end{definition}

\subsection{Continua, arcs, trees, and dendrites}

\begin{definition}[Continuum, arc, and simple closed curve]
A {continuum} is a nonempty compact connected metric space.  An
{arc} is a space homeomorphic to the interval $[0,1]$.  A
{simple closed curve}, also called a topological circle, is a space
homeomorphic to the unit circle $S^1$.
\end{definition}

\begin{definition}[Finite tree]
A {finite tree}, or simply a {tree} in this article, is a continuum
that is the union of finitely many arcs with pairwise finite intersections and
contains no simple closed curve.  Every finite tree is planar, so after a
rescaling it may be embedded in any disk of positive radius.
\end{definition}

\begin{definition}[Dendrite]
A {dendrite} is a locally connected continuum containing no simple
closed curve.  Every finite tree is therefore a dendrite.   Standard definitions and equivalent characterizations are
given in \cite[Chapter X]{Nadler1992}.
\end{definition}

\subsection{Retracts and absolute retracts}

\begin{definition}[Retraction and retract]
Let $A$ be a subspace of a space $X$.  A continuous map
$
 r:X\rightarrow A
$
is a {retraction} if $r(a)=a$ whenever $a\in A.$
 In that case $A$ is called a
retract of $X$.
\end{definition}

\begin{definition}[Absolute retract]
A compact metric space $A$ is an {absolute retract for compact metric
spaces} if, whenever $A$ is embedded as a closed subspace of a compact metric
space $X$, the embedded copy of $A$ is a retract of $X$.  
\end{definition}

\begin{theorem}[Borsuk's dendrite theorem]\label{thm:dendrite-ar}
Every dendrite is an absolute retract for the class of compact metric spaces.
\end{theorem}

This classical theorem is recorded in
\cite[Corollary 13.5, p.~138]{Borsuk1967}.  Consequently, if a finite tree
$T$ is embedded in a closed disk $D$, its image $A\subset D$ admits a
retraction
$
 r:D\rightarrow A.$
This consequence is the
only retract-theoretic fact needed in our proof.

\section{Proof of Theorem 1}

We now recall the topological input needed to prove Theorem 1 that we borrow from \cite{OversteegenRogers1980,OversteegenRogers1982}. See also \cite{HernandezHoehn2018} for an excellent exposition.

\begin{proposition}[Oversteegen and Rogers]\label{prop:tree-ladder}
We have, for every $n\in \Z_+$,	 a finite  tree $T_n$ and continuous maps
	$
	f_n,g_n:T_{n+1}\rightarrow T_n, 
	$
	such that
	\begin{equation*}
		f_n\circ g_{n+1}=g_n\circ f_{n+1},
	\end{equation*}
	and
	$
	\Coin(f_0,g_0)=\varnothing.
	$
\end{proposition}

By \cref{thm:dendrite-ar}, every embedded finite tree is a retract of any compact metric space containing it as a closed subspace.  In particular, an embedded copy of $T_n$ in a closed disk admits a continuous retraction from that disk.  This is the extension mechanism we use. 

\begin{proof}

Take the trees and maps supplied by \cref{prop:tree-ladder}. For every nonnegative integer $n$, choose
$
\rho_n:=2^{-n-2},
\ n\in\Z_+
$
and let $D_n$ be the closed disk of radius $\rho_n$ centered at the origin in $\R^2$.
Since every finite tree is planar, choose a topological embedding
$ e_n:T_n\rightarrow A_n$
where
$
A_n:=e_n(T_n)\subset D_n.
$
We transport the maps to the embedded copies by setting
\begin{equation*}\label{}
	\widetilde{f}_n:=e_n\circ f_n\circ e_{n+1}^{-1},
	\qquad
		\widetilde{g}_n:=e_n\circ g_n\circ e_{n+1}^{-1}.
\end{equation*}
Then
\begin{equation}\label{eq:abstract-no-coincidence}
		\widetilde{f}_n\circ	\widetilde{g}_{n+1}
	=	\widetilde{g}_n\circ	\widetilde{f}_{n+1},
	\qquad \text{and}  \qquad
	\Coin(	\widetilde{f}_0,	\widetilde{g}_0)=\varnothing.
\end{equation}
Because $A_n$ is a dendrite, there is a continuous retraction
$
r_n:D_n\rightarrow A_n.
$
The set and the space  that constitute the counterexample are 
\begin{equation*}
	K=\prod_{n=0}^{\infty}D_n
	\subset \ell^2(\Z_+;\R^2).
\end{equation*}
As a direct sum of 2 dimensional Hilbert spaces,  $\ell^2(\Z_+;\R^2)$ is a separable Hilbert space. As $\rho_n, \ n\in \Z_+$ is square summable, $K$ lies entirely in this space. Convexity of factors immediately  yield convexity of $K$.
We next prove  compactness of $K$ with respect to the norm topology of its ambient space, by proving sequential compactness.  

 Let $\{x^{m}\}_{m\geq1}$ be a sequence in $K$, where
$
 x^{m}=(x_n^{m})_{n\in \Z_+}.
$ 
We note that $x_n^m\in \R^2$, for any $m,n.$  Because each $D_n$ is compact, a diagonal argument gives a subsequence, again denoted by $(x^{m})$, and points $x_n\in D_n$ such that
$
 x_n^{(m)}\rightarrow x_n$
 in $\R^2
$
for every fixed $n$. Put $x=(x_n)_{n\geq0}\in K$. Given $\varepsilon>0$, choose $N$ so large that
\begin{equation}\label{eq:tail-choice}
 \sum_{n>N}\rho_n^2<{\varepsilon^2}/{8}.
\end{equation}
For sufficiently large $m$,
\begin{equation*}
 \sum_{n=0}^{N}\lVert x_n^{(m)}-x_n\rVert_2^2
 <{\varepsilon^2}/{2},
\end{equation*}
 where $\|\cdot\|_2$ is the usual Euclidean distance on $\R^2$.  Moreover,  as $x_n^m,x_n\in D_n$ for any $m,n$, we have
\begin{equation*}
 \sum_{n>N}\lVert x_n^{(m)}-x_n\rVert_n^2
 \leq 4\sum_{n>N}\rho_n^2
 <{\varepsilon^2}/{2}.
\end{equation*}
Thus $x^{(m)}\to x$ in $ \ell^2(\Z_+;\R^2)$. Hence $K$ is sequentially compact and therefore compact.

To introduce our maps $F,G$ we first  define 
\begin{align*}
 \alpha_n&:=	\widetilde{f}_n\circ r_{n+1}:D_{n+1}\rightarrow A_n\subset D_n,\\
 \beta_n&:=	\widetilde{g}_n\circ r_{n+1}:D_{n+1}\rightarrow A_n\subset D_n.
\end{align*}
Since $r_{n+1}$ is the identity on $A_{n+1}$, the ladder identities remain exact. Indeed,
\begin{align}
 \alpha_n\circ\beta_{n+1}
 =	\widetilde{f}_n\circ r_{n+1}\circ 	\widetilde{g}_{n+1}\circ r_{n+2}
 =	\widetilde{f}_n\circ 	\widetilde{g}_{n+1}\circ r_{n+2}
 =	\widetilde{g}_n\circ 	\widetilde{f}_{n+1}\circ r_{n+2}
 =\beta_n\circ\alpha_{n+1}.
 \label{eq:extended-ladder}
\end{align}
Also,
\begin{equation}\label{eq:extended-no-coincidence}
 \alpha_0(z)\neq\beta_0(z)
 \qquad\text{for every }z\in D_1.
\end{equation}
For if equality held, then with $u=r_1(z)\in A_1$ we would have $	\widetilde{f}_0(u)=	\widetilde{f}_0(u)$, contradicting \eqref{eq:abstract-no-coincidence}.
Now we can define $F,G.$  For $x=(x_0,x_1,x_2,\ldots)\in K$, 
\begin{align*}
 F(x):=\bigl(\alpha_0(x_1),\alpha_1(x_2),\alpha_2(x_3),\ldots\bigr),  \qquad  \qquad  G(x):=\bigl(\beta_0(x_1),\beta_1(x_2),\beta_2(x_3),\ldots\bigr).
 \label{eq:G-definition}
\end{align*} 
Since $\alpha_n(x_{n+1}),\beta_n(x_{n+1})\in D_n$, both maps take $K$ into itself. Since $\alpha_0(x_1)\neq \beta_0(x_1),$ we have $F(x)\neq G(x).$ So by their definition $F,G$ never coincide. It remains to verify that  they are continuous and that they commute.

We verify continuity of $F$; the proof for $G$ is identical. Suppose $x^{m}\to x$ in $K$. For every  $n\in \Z_+$, the coordinate maps $\alpha_{n},\beta_n$ are continuous, as they are compositions of continuous maps.
Choose $N$ as in \eqref{eq:tail-choice}. On the first $N+1$ coordinates, convergence is therefore uniform in the finite-sum sense, while on the tail, because the range of $\alpha_{n}$ is $D_n$, one has
\begin{equation*}
 \sum_{n>N}
 \left\lVert
 \alpha_n(x_{n+1}^{m})-\alpha_n(x_{n+1})
 \right\rVert_n^2
 \leq 4\sum_{n>N}\rho_n^2.
\end{equation*}
It follows that $F(x^{(m)})\to F(x)$ in $\ell^2(\Z_+;\R^2)$.

Next we verify that $F$ and $G$ commute.  For every $x\in K$ and every $n\in \Z_+$, \eqref{eq:extended-ladder} gives
\begin{align*}
 (F\circ G)(x)_n
 =\alpha_n\bigl(\beta_{n+1}(x_{n+2})\bigr)
 =\beta_n\bigl(\alpha_{n+1}(x_{n+2})\bigr)
 =(G\circ F)(x)_n.
\end{align*}
Hence $F\circ G=G\circ F$. This concludes the proof.
\end{proof}

\end{document}